\documentclass{amsart}
\usepackage{amsmath, amssymb, amsthm}
\usepackage[all]{xy}

\theoremstyle{definition}
\newtheorem{theo}{Theorem}[section]

\newtheorem{lemma}[theo]{Lemma}
\newtheorem{prop}[theo]{Proposition}

\newtheorem{rema}[theo]{Remark}

\makeatletter
\@addtoreset{equation}{section} 
\makeatother

\title[quasi-morphisms on groups of symplectomorphisms of the disk]{The translation number and quasi-morphisms on groups of symplectomorphisms of the disk}
\author{Shuhei Maruyama}
\address{Graduate School of Mathematics, Nagoya University, Japan}
\email{m17037h@math.nagoya-u.ac.jp}

\begin{document}

\begin{abstract}
  On groups of symplectomorphisms of the disk, we construct
  two homogeneous quasi-morphisms which relate to the
  Calabi invariant and the flux homomorphism
  respectively.
  We also show the relation between the quasi-morphisms
  and the translation number introduced by Poincar\'{e}.
\end{abstract}

\maketitle

\section{Introduction}

A quasi-morphism on a group $\Gamma$ is a function
$\phi : \Gamma \to \mathbb{R}$
such that the value
\[
  \sup_{\gamma_1, \gamma_2 \in \Gamma}
  | \phi(\gamma_1\gamma_2) - \phi(\gamma_1) - \phi(\gamma_2)|
\]
is bounded.
A quasi-morphism $\phi$ is called homogeneous
if the condition $\phi(\gamma^n) = n\phi(\gamma)$ holds
for any $\gamma \in \Gamma$ and $n \in \mathbb{Z}$.
Let $Q(\Gamma)$ denote the $\mathbb{R}$-vector space of
homogeneous quasi-morphisms on the group $\Gamma$.
Given a quasi-morphism $\phi$,
we obtain the homogeneous quasi-morphism $\overline{\phi}$
associated to $\phi$ by
\[
  \overline{\phi}(g) = \lim_{n \to \infty} \frac{\phi(g^n)}{n}.
\]
This map $\overline{\phi}$ is called the homogenization of $\phi$.

Let $D = \{ (x,y) \in \mathbb{R}^2 \mid x^2 + y^2 \leq 1 \}$
be the unit disk in $\mathbb{R}^2$ and
$\omega = dx \wedge dy$ be the standard symplectic form on $D$.
Let $G = {\rm Symp}(D)$ be the group of symplectomorphisms of $D$
(which may not be the identity on the boundary $\partial D$).
In the present paper, we construct
a homogeneous quasi-morphism on $G$.
Let $\eta$ be a $1$-form on $D$ satisfying $d\eta = \omega$.
The map $\tau_{\eta}:G \to \mathbb{R}$ is defined by
\[
  \tau_{\eta}(g) = \int_D g^*\eta \wedge \eta.
\]
Let $G_{\rm rel}$ denote the kernel of the homomorphism
$G \to {\rm Diff}_{+}(S^1)$,
where ${\rm Diff}_{+}(S^1)$ denotes the group of
orientation preserving diffeomorphisms of the circle.
Then the map $\tau$ coincides with the Calabi invariant
on $G_{\rm rel}$.
Although the Calabi invariant ${\rm Cal} : G_{\rm rel} \to \mathbb{R}$
is a homomorphism,
the map $\tau_{\eta} : G \to \mathbb{R}$ is not a homomorphism.
However, this map $\tau_{\eta}$ gives rise to a quasi-morphism.
Thus, by the homogenization,
we have the homogeneous quasi-morphism
$\overline{\tau_{\eta}}$.
Since $\overline{\tau_{\eta}}$ is independent
of the choice of $\eta$, we simply denote it by $\overline{\tau}$.
This $\overline{\tau}$
is the main object of the present paper.

It is known that the Calabi invariant
${\rm Cal} : G_{\rm rel} \to \mathbb{R}$
cannot be extended to a homomorphism $G \to \mathbb{R}$
(see Tsuboi\cite{tsuboi00}).
However, the Calabi invariant \textit{can be} extended to a
homogeneous quasi-morphism on $G$.
Indeed, we will show in Proposition \ref{tau_surj} that
the homogeneous quasi-morphism
$\overline{\tau} : G \to \mathbb{R}$ gives rise to
an extension of the Calabi invariant.
There is another extension $R$ of the Calabi invariant,
which is introduced by Tsuboi\cite{tsuboi00}
(see also Banyaga\cite{banyaga78}).
This extension $R$ is defined as a homomorphism
to $\mathbb{R}$
from the universal covering group
$\widetilde{G}$ of $G$ by
\[
  R([g_t]) = \int_0^1 \left( \int_D f_{X_t}\omega \right) dt.
\]
Here $g_t$ is a path in $G$ and $f_{X_t}$ is the
Hamiltonian function associated to $g_t$
(see Section $4$).
Then, it is natural to ask what the relation between two extensions
$\tau$ and $R$ of the Calabi invariant is.
The following theorem answers this.

\begin{theo}[Theorem \ref{main_thm_1}]\label{main_thm_1_intro}
  Let $p: \widetilde{G} \to G$ be the projection.
  Then, we have
  \[
    p^*\overline{\tau} + 2R = \pi^2 \widetilde{\rm rot}
    : \widetilde{G} \to \mathbb{R}.
  \]
  Here the map $\widetilde{\rm rot}:\widetilde{G} \to \mathbb{R}$
  is the pullback of Poincar\'{e}'s translation number
  by the surjection
  $\widetilde{G} \to \widetilde{{\rm Diff}_{+}(S^1)}$.
\end{theo}

Let $G_o = \{ g \in G \mid g(o) = o \}$ be the subgroup
of $G$ consisting of symplectomorphisms which preserve the origin
$o = (0,0) \in D$.
On the group $G_o$, we also construct a homogeneous quasi-morphism
$\overline{\sigma} = \overline{\sigma}_{\eta, \gamma}
: G_o \to \mathbb{R}$, where
$\sigma_{\eta, \gamma} : G_o \to \mathbb{R}$ is defined by
\[
  \sigma_{\eta, \gamma}(g) = \int_{\gamma}g^*\eta - \eta.
\]
Here the symbol $\gamma$ is a path from the origin
to a point on the boundary.
Let $\widetilde{G_o}$ be the universal covering group of
$G_o$.
By using the homomorphism
$S:\widetilde{G_o} \to \mathbb{R}$ introduced in Section
\ref{section_flux},
we describe the relation between $\overline{\sigma}$ and the
translation number,
which is similar to Theorem \ref{main_thm_1_intro}.
\begin{theo}[Theorem \ref{main_thm_2}]
  Let $p: \widetilde{G_o} \to G_o$ be the projection.
  Then, we have
  \[
    p^*\overline{\sigma} - S = \pi \widetilde{\rm rot}
    : \widetilde{G_o} \to \mathbb{R}.
  \]
  Here the map $\widetilde{\rm rot}:\widetilde{G_o} \to \mathbb{R}$
  is the pullback of the translation number
  by the surjection
  $\widetilde{G_o} \to \widetilde{{\rm Diff}_{+}(S^1)}$.
\end{theo}

The coboundary of the translation number $\widetilde{\rm rot}$
gives the canonical Euler cocycle (Matsumoto \cite{matsumoto86}).
Similarly, the coboundary of homogeneous quasi-morphisms
$\overline{\tau}$ and $\overline{\sigma}$
also give cocycles which represents the bounded Euler class
of ${\rm Diff}_+(S^1)$
(Propositions \ref{tau_and_bdd_euler} and \ref{sigma_bar}).

By comparing the two homogeneous quasi-morphisms
$\overline{\tau}$ and $\overline{\sigma}$,
we obtain the following theorem.
\begin{theo}[Theorem \ref{main_thm_3}]
  The difference
  $\overline{\tau} - \pi \overline{\sigma}: G_o \to \mathbb{R}$
  is a continuous surjective homomorphism.
\end{theo}

Note that, in this paper,
we assume the notation of group cohomology and bounded cohomology in
\cite{calegari09}.


\subsection*{Acknowledgements}
The author would like to thank Professor Hitoshi Moriyoshi for
his helpful advice.
He also thanks Morimichi Kawasaki, who told him that there is another
extension $R$ of the Calabi invariant and suggested to investigate
a connection between $R$ and the quasi-morphism $\overline{\tau}$
constructed in this paper.
He also thanks Professor Masayuki Asaoka for his comments.

\section{The Calabi invariant case}\label{section_tau_R}
\subsection{Calabi invariant and the quasi-morphism $\overline{\tau}$}
Let $D = \{ (x, y) \in \mathbb{R}^2 \mid x^2 + y^2 \leq 1 \}$
be the unit disk with the standard symplectic form
$\omega = dx \wedge dy$.
Let $G = {\rm Symp}(D)$ denote the symplectomorphism group of $D$
and ${\rm Diff}_{+}(S^1)$
the orientation preserving diffeomorphism group
of the unit circle $S^1 = \partial D$.
Then the homomorphism
$\rho : G \to {\rm Diff}_{+}(S^1)$
is surjective(see Tsuboi\cite{tsuboi00}).
Thus we have an exact sequence
\[
  1 \longrightarrow G_{\rm rel} \longrightarrow G
  \overset{\rho}{\longrightarrow} {\rm Diff}_{+}(S^1) \longrightarrow 1,
\]
where the group $G_{\rm rel}$ is the kernel of the map
$\rho : G \to {\rm Diff}_{+}(S^1)$.

The Calabi invariant ${\rm Cal} : G_{\rm rel} \to \mathbb{R}$ is defined by
\begin{equation}\label{def_calabi}
  {\rm Cal}(h)
  = \int_{D} h^*\eta \wedge \eta
\end{equation}
where $\eta$ is a $1$-form satisfying $d\eta = \omega$.
The Calabi invariant
${\rm Cal}$ is a surjective homomorphism
and is independent of the choice of $\eta$ (see Banyaga\cite{banyaga78}).
On the group $G$,
the map $\tau_{\eta} : G \to \mathbb{R}$ is defined in the same way
as in (\ref{def_calabi}), that is, we put
\[
  \tau_{\eta}(g) = \int_{D} g^*\eta \wedge \eta.
\]
Note that the map $\tau_{\eta}$ is {\it not} a homomorphism and
{\it does} depend on the choice of $\eta$.
In \cite{moriyoshi16}, for $\lambda = (xdy - ydx)/2$, Moriyoshi proved
the transgression formula
\begin{align}\label{transgression}
  {\rm Cal}(h) &= \tau_{\lambda}(h)
  \hspace{5mm} (h \in G_{\rm rel}) \nonumber \\
  - \delta \tau_{\lambda}(g, h) &=
  \pi^2 \chi(\rho(g), \rho(h)) + \pi^2/2 \hspace{5mm} (g, h \in G).
\end{align}
Here $\delta$ is the coboundary operator
of group cohomology and
the symbol $\chi$ is a bounded $2$-cocycle defined in
Moriyoshi\cite{moriyoshi16},
which represents the bounded Euler class
$e_b \in H_b^2({\rm Diff}_{+}(S^1);\mathbb{R})$.
Since the cocycle $\chi$ is bounded,
the map $\tau_{\lambda} : G \to \mathbb{R}$ is a quasi-morphism.
Moreover, since the function $\tau_{\eta} - \tau_{\lambda}$ is
bounded
for any $1$-form $\eta$ satisfying
$d\eta = \omega$, the map $\tau_{\eta}$ is a quasi-morphism
for any $\eta$ and the homogenizations of
$\tau_{\eta}$ and $\tau_{\lambda}$ coincide.
Thus we simply denote by $\overline{\tau}$
the homogenization of $\tau_{\eta}$.

\begin{prop}\label{tau_surj}
  The homogenization $\overline{\tau} : G \to \mathbb{R}$ is an
  extension of the Calabi invariant,
  that is, $\overline{\tau}|_{G_{\rm rel}} = {\rm Cal}$.
  In particular, the map $\overline{\tau}$ is a
  surjective
  homogeneous quasi-morphism.
\end{prop}

\begin{proof}
  For $h \in G_{\rm rel}$, we have
  \[
    \overline{\tau}(h) = \lim_{n \to \infty}\frac{\tau_{\eta}(h^n)}{n}
    = \lim_{n \to \infty}\frac{{\rm Cal}(h^n)}{n}
    = \lim_{n \to \infty}\frac{n{\rm Cal}(h)}{n}
    = {\rm Cal}(h).
  \]
  Since the Calabi invariant is surjective, the homogenization
  $\overline{\tau}$ is also surjective.
\end{proof}

The homogeneous quasi-morphism
$\overline{\tau}$ relates to the bounded Euler class as follows.
\begin{prop}\label{tau_and_bdd_euler}
  The bounded cohomology class $[\delta \overline{\tau}] \in H_b^2(G;\mathbb{R})$
  is equal to $-\pi^2$ times
  the pullback $\rho^* e_b$ of the bounded Euler class
  $e_b$.
\end{prop}

\begin{proof}
  Recall that the difference between a quasi-morphism and its
  homogenization is a bounded function.
  Thus we have $\delta \tau_{\lambda} - \delta \overline{\tau} = \delta b$
  where $b = \tau_{\lambda} - \overline{\tau}$ is a bounded function.
  This implies that the bounded cohomology class
  $[\delta \tau_{\eta}]$ coincides with $[\delta \overline{\tau}]$.
  Moreover, the class $[\delta \tau_{\lambda}]$ is equal to the
  pullback $\rho^* e_b$ up to non-zero constant multiple
  because of the transgression formula
  (\ref{transgression}).
\end{proof}

\subsection{Two extensions $\overline{\tau}$ and $R$ of the Calabi invariant}
By Proposition \ref{tau_surj}, the homogeneous quasi-morphism
$\overline{\tau}: G \to \mathbb{R}$ is considered as
an extension of the Calabi invariant.
There is another extension $R$ of the Calabi invariant,
which is introduced by Tsuboi\cite{tsuboi00}
(see also Banyaga\cite{banyaga78}).
This extension is defined as
a homomorphism $R : \widetilde{G} \to \mathbb{R}$,
where the group $\widetilde{G}$ is the universal covering group of
$G$ with respect to the $C^{\infty}$-topology.
In this section, we investigate the relation between these
two extensions $\overline{\tau}$ and $R$.

We recall the definition of the homomorphism $R$.
Let $\mathcal{L}_{\omega}(D)$ be the set of divergence free vector
fields which are tangent to the boundary.
For any vector field $X$ in $\mathcal{L}_{\omega}(D)$,
there is a unique function $f_X : D \to \mathbb{R}$
such that $i_X \omega = df_X$ and $f_X|_{\partial D} = 0$.
For any path $g_t$ in $G$, we define the time-dependent vector field
$X_t$ by $X_t = (\partial g_t / \partial t)\circ g_t^{-1}$.
Since $g_t$ is a symplectomorphism for any $t \in [0,1]$,
the vector field $X_t$ is
in $\mathcal{L}_{\omega}(D)$.
Then the map $R:\widetilde{G} \to \mathbb{R}$ is defined
by
\[
  R([g_t]) = \int_0^1 \left( \int_D f_{X_t}\omega \right) dt.
\]
This map $R$ is a well-defined homomorphism
(see Banyaga\cite{banyaga78}).

We reproduce the following lemma,
which is essentially proved in
Tsuboi\cite[Lemme 1.5]{tsuboi00}.
\begin{lemma}
  Let $g_t$ be a path in $G$ such that $g_0 = {\rm id}$ and $X_t$
  the time-dependent vector field defined by
  $X_t = (\partial g_t / \partial t)\circ g_t^{-1}$,
  then
  \begin{equation}\label{calabi_r_rot}
    \tau_{\eta}(g_1) + 2R([g_t])
    =
    \int_{\partial D} \left( \int_0^1 g_t^* (i_{X_t} \eta) dt \right)
    \eta.
  \end{equation}
  In particular, for a path $h_t$ in $G_{\rm rel}$ such that
  $h_0 = {\rm id}$,
  we have ${\rm Cal}(h_1) = -2R([h_t])$.
\end{lemma}

Let $\widetilde{{\rm Diff}_{+}(S^1)}$ denote the universal
covering of ${\rm Diff}_{+}(S^1)$.
Note that , in this paper, we identify the circle $S^1$
with the quotient $\mathbb{R}/2\pi \mathbb{Z}$.
We consider an element
$\widetilde{\gamma} \in \widetilde{{\rm Diff}_{+}(S^1)}$
as an orientation preserving
diffeomorphism of $\mathbb{R}$ satisfying
$\widetilde{\gamma}(\theta + 2\pi) = \widetilde{\gamma}(\theta) + 2\pi$
for any $\theta \in \mathbb{R}$.
Let $\varphi_t$ be the path in ${\rm Diff}_{+}(S^1)$ defined by
$\varphi_t = g_t|_{\partial D}$.
Let $\xi_t$ be the time-dependent vector field defined by
$\xi_t = (\partial \varphi_t / \partial t)\circ \varphi_t^{-1}$.
Let $\widetilde{\varphi_t} \in \widetilde{{\rm Diff}_{+}(S^1)}$
be the lift of $\varphi_t$ such that $\widetilde{\varphi_0} = {\rm id}$.
Note that $\lambda = (xdy - ydx)/2 = (r^2 d\theta)/2$
where $(r, \theta) \in D$ is the polar coordinates.
Then the right-hand side of the equality (\ref{calabi_r_rot})
can be written as
\begin{align}\label{i_X_t_eta_integral}
  \int_{\partial D} \left( \int_0^1 g_t^* (i_{X_t} \lambda) dt \right)
  \lambda
  &=
  \frac{1}{4}\int_{S^1}\left( \int_0^1 \varphi_t^*
  (i_{\xi_t}d\theta)dt \right) d\theta\\ \nonumber
  &=
  \frac{1}{4}\int_0^{2\pi} \left(
  \int_0^1 \frac{\partial \widetilde{\varphi_t}}{\partial t}
  dt \right) d\theta
  =
  \frac{1}{4}\int_0^{2\pi} (\widetilde{\varphi_1}(\theta) - \theta)
  d\theta. \nonumber
\end{align}
Let us define a map $f : \widetilde{{\rm Diff}_{+}(S^1)} \to \mathbb{R}$
by $f(\widetilde{\varphi})
  =
  \frac{1}{4\pi^2}\int_0^{2\pi} (\widetilde{\varphi}(\theta) - \theta)
  d\theta$.
Then we have
\begin{equation}\label{tau_R_rho}
  \tau_{\lambda}(g_1) + 2R([g_t]) =
  \pi^2f(\widetilde{\varphi_1}).
\end{equation}
Note that, for any $\widetilde{\varphi}, \widetilde{\psi}$ in
$\widetilde{{\rm Diff}_{+}(S^1)}$, the inequality
$|\widetilde{\varphi}\widetilde{\psi}(\theta)
- \widetilde{\varphi}(\theta) - \widetilde{\psi}(\theta) + \theta|
< 4\pi$
holds.
This implies that the map $f$ is a quasi-morphism.
Let $\overline{f}$
be the homogenization of $f$.
By taking the homogenizations of the both sides of the
equality (\ref{tau_R_rho}),
we have
\begin{equation}\label{tau_R_rho_bar}
  \overline{\tau}(g_1) + 2R([g_t])
  =
  \pi^2\overline{f}(\widetilde{\varphi_1}).
\end{equation}

To explain the map
$\overline{f} : \widetilde{{\rm Diff}_{+}(S^1)} \to \mathbb{R}$,
we recall the translation number introduced by Poincar\'{e}\cite{poincare85}.
The translation number is a homogeneous quasi-morphism
$\widetilde{\rm rot} : \widetilde{{\rm Diff}_{+}(S^1)} \to \mathbb{R}$
defined by
$\widetilde{\rm rot}(\widetilde{\varphi})
=\displaystyle
\lim_{n\to \infty}\frac{\widetilde{\varphi}^n(0)}{2\pi n}$.
Note that, in this paper, we identify the circle $S^1$ with the quotient
$\mathbb{R}/2\pi \mathbb{Z}$.

\begin{prop}\label{rho_translation}
  The homogeneous quasi-morphism
  $\overline{f} : \widetilde{{\rm Diff}_{+}(S^1)} \to \mathbb{R}$
  coincides with the translation number.
\end{prop}

\begin{proof}
  Since the sequence
  $\left\{ \frac{\widetilde{\varphi}^n(x) - x}{n} \right\}_{n}$
  converges uniformly to the constant function
  $\displaystyle \lim_{n \to \infty} \widetilde{\varphi}^n(0)/n$
  on the interval $[0,2\pi]$,
  We have
  \[
    \overline{f}(\widetilde{\varphi})
    =
    \frac{1}{4\pi^2}\lim_{n\to \infty}
    \int_0^{2\pi} \frac{\widetilde{\varphi}^n(x) - x}{n}dx
    =
    \frac{1}{4\pi^2}\int_0^{2\pi}
    \lim_{n\to \infty}\frac{\widetilde{\varphi}^n(0)}{n}dx
    = \widetilde{\rm rot}(\widetilde{\varphi}).
  \]
\end{proof}

By Proposition \ref{rho_translation} and
equality (\ref{tau_R_rho_bar}),
we obtain the following theorem.
\begin{theo}\label{main_thm_1}
  Let $p: \widetilde{G} \to G$ be the projection.
  Then we have
  \[
    p^*\overline{\tau} + 2R = \pi^2 \widetilde{\rm rot}
    : \widetilde{G} \to \mathbb{R}.
  \]
  Here the map $\widetilde{\rm rot}:\widetilde{G} \to \mathbb{R}$
  is the pullback of the translation number
  by the surjection
  $\widetilde{G} \to \widetilde{{\rm Diff}_{+}(S^1)}$.
\end{theo}

Poincar\'{e}'s translation number descends to the map
${\rm rot}:{\rm Diff}_{+}(S^1) \to \mathbb{R}/\mathbb{Z}$ and this is
called Poincar\'{e}'s rotation number.
The homomorphism
$2R/\pi^2:\widetilde{G} \to \mathbb{R}$ also decends to
the homomorphism $\underline{R}: G \to \mathbb{R}/\mathbb{Z}$
(see Tsuboi\cite[Corollary 2.9]{tsuboi00}).

\begin{theo}\label{theorem:tau_R_rotation}
  Let $\underline{\tau} : G \to \mathbb{R}/\mathbb{Z}$ be the
  composition of the homogeneous quasi-morphism
  $\overline{\tau}/\pi^2 : G \to \mathbb{R}$
  and the projection $\mathbb{R} \to \mathbb{R}/\mathbb{Z}$,
  then
  \[
    \underline{\tau} + \underline{R} = {\rm rot}.
  \]
  Here the ${\rm rot} : G \to \mathbb{R}/\mathbb{Z}$
  is the pullback of the rotation number by
  the projection $G \to {\rm Diff}_{+}(S^1)$.
\end{theo}

\section{The flux homomorphism case}\label{section_flux}
\subsection{The flux homomorphism and the quasi-morphism $\overline{\sigma}$}
Let us consider the subgroup
\[
  G_o = \{ g \in G \mid g(o) = o \in D \}
\]
of $G$.
Put $G_{o,{\rm rel}} = G_{\rm rel} \cap G_o$.
Then the following sequence of groups
\[
  1 \longrightarrow G_{o,{\rm rel}} \longrightarrow G_o
  \overset{\rho}{\longrightarrow} {\rm Diff}_{+}(S^1) \longrightarrow 1
\]
is an exact sequence.
On the group $G_{o,{\rm rel}}$,
the Calabi invariant is defined
as the restriction ${\rm Cal}|_{G_{o,{\rm rel}}}: G_{o,{\rm rel}} \to \mathbb{R}$.
In \cite{maruyama19}
the author studied a version of
flux homomorphism defined on $G_{o, {\rm rel}}$
which is denoted by ${\rm Flux}_{\mathbb{R}}$.
This flux homomorphism ${\rm Flux}_{\mathbb{R}}$ is defined by
\begin{equation*}
  {\rm Flux}_{\mathbb{R}}(h) = \int_{\gamma} h^*\eta - \eta
\end{equation*}
where $\gamma$ is a path from the origin $o$ to a point on the
boundary $\partial D$.
Note that the flux homomorphism is a surjective homomorphism
and is independent of the choice of $\eta$ and $\gamma$.

As in the case of Calabi invariant, the flux homomorphism can be
extended to the group $G_o$, that is,
we define the map $\sigma_{\eta,\gamma}:G_o \to \mathbb{R}$ by
\begin{equation*}
  \sigma_{\eta,\gamma}(g) = \int_{\gamma} g^*\eta - \eta.
\end{equation*}
The following transgression formula
\begin{align}\label{transgression_flux}
  {\rm Flux}_{\mathbb{R}}(h) &= \sigma_{\eta, \gamma}(h)
  \hspace{5mm} (h \in G_{o,{\rm rel}}) \nonumber \\
  - \delta \sigma_{\eta, \gamma}(g, h) &=
  \pi \xi(\rho(g), \rho(h)) \hspace{5mm} (g, h \in G_o),
\end{align}
holds,
where $\xi \in C^2({\rm Diff}_{+}(S^1);\mathbb{R})$ is an
Euler cocycle (see \cite{maruyama19}, where, in \cite{maruyama19},
the map $\sigma_{\eta,\gamma}$
is denoted by $\tau$ and
the Euler cocycle $\xi$ is
denoted by $\chi$).
Since $\xi$ is bounded, the map $\sigma_{\eta, \gamma}$ is
a quasi-morphism.
Let $\overline{\sigma}$ denote the homogenization of
$\sigma_{\eta, \gamma}$.
By arguments similar to those in Section \ref{section_tau_R},
we obtain the following proposition.

\begin{prop}\label{sigma_bar}
  \begin{enumerate}
    \item The homogenization $\overline{\sigma}:G_o \to \mathbb{R}$
    is independent of the choice of $\eta$ and $\gamma$.
    \item The homogenization $\overline{\sigma} : G_o \to \mathbb{R}$
    is an extension of the flux homomorphism.
    In particular, $\overline{\sigma}$ is a
    surjective homogeneous quasi-morphism.
    \item The bounded cohomology class $[\delta \overline{\sigma}]$ is equal to
    $-\pi$ times the class $\rho^* e_b$, where $e_b$ is the bounded
    Euler class.
  \end{enumerate}
\end{prop}

\begin{rema}
  For an inner point $a \in D$,
  put $G^{a} = \{ g \in G \mid g(a) = a \}$.
  We can define the homogeneous
  quasi-morphism $\overline{\sigma}_a : G^a \to \mathbb{R}$
  in the same way.
  We can also show that $[\delta \overline{\sigma}_a]
  = -\pi\rho^*e_b$.
  Thus, for inner points $a, b \in D$,
  we have a homomorphism
  \[
    \overline{\sigma}_a - \overline{\sigma}_b
    : G^a \cap G^b \to \mathbb{R}
  \]
  and this is equal to the action difference defined in
  Polterovich\cite{polterovich02}(see also \cite{gal_kedra11}).
\end{rema}

\subsection{Two extensions $\overline{\sigma}$ and $S$ of the flux homomorphism}
Let $\widetilde{G_o}$ be the universal covering group of $G_o$
with respect to the $C^{\infty}$-topology.
In this section, we introduce a homomorphism
$S: \widetilde{G_o} \to \mathbb{R}$
and show that the difference of $\overline{\sigma}$ and $S$ is equal to the
translation number.

For a path $g_t$ in $G_o$ such that $g_0 = {\rm id}$,
the time-dependent vector field $X_t$
is defined as in Section \ref{section_tau_R}.
Then we put
\begin{equation}
  S(g_t) = \int_0^1 \int_{\gamma} i_{X_t} \omega dt,
\end{equation}
where $\gamma:[0,1] \to D$ is a path from the origin $o \in D$
to a point on the boundary $\partial D$.
Take the time-dependent
$C^{\infty}$-function $f_t : D \to \mathbb{R}$
satisfying $i_{X_t}\omega = df_t$
and $f_t(o) = 0$.
Then we have
\begin{equation*}
  S(g_t) = \int_0^1 \int_{\gamma} i_{X_t} \omega dt
  = \int_0^1 \int_{\gamma} df_t dt
  = \int_0^1 f_t(\gamma(1)) dt.
\end{equation*}
Note that, for any $t \in [0, 1]$, the restriction
$f_t|_{\partial D} :\partial D \to \mathbb{R}$ is a
constant function.
This implies that the function $S$ is independent of the
choice of $\gamma$.


\begin{lemma}
  Let $g_t$ be a path in $G_o$ such that $g_0 = {\rm id}$ and $X_t$
  the time-dependent vector field defined by
  $X_t = (\partial g_t / \partial t)\circ g_t^{-1}$,
  then
  \begin{equation}
    \sigma_{\eta, \gamma}(g_1) - S(g_t)
    =
    \int_0^1 (g_t^*(i_{X_t}\eta))(\gamma(1))dt.
  \end{equation}
\end{lemma}

\begin{proof}
  Note that the identity
  \[
    g_1^* \eta - \eta
    =
    d\left( \int_0^1 g_t^*f_t dt +
    \int_0^1 g_t^* (i_{X_t} \eta) dt \right).
  \]
  holds.
  Thus we have
  \begin{align}\label{calculus}
    \nonumber
    \sigma_{\eta, \gamma}(g_1)
    &=
    \int_{\gamma} g_1^* \eta - \eta
    =
    \int_{\gamma}d\left( \int_0^1 g_t^*f_tdt +
    \int_0^1 g_t^* (i_{X_t} \eta) dt \right)\\
    &=
    \left( \int_0^1 (g_t^*f_t)(\gamma(1))dt +
    \int_0^1 (g_t^* (i_{X_t} \eta))(\gamma(1)) dt \right)\\
    &\nonumber \hspace{1cm} -
    \left( \int_0^1 (g_t^*f_t) (\gamma(0))dt +
    \int_0^1 (g_t^* (i_{X_t} \eta))(\gamma(0)) dt \right).
  \end{align}
  Since $(g_t^* f_t)(\gamma(0)) = 0$ and $X_t(\gamma(0)) = 0$
  for any $t \in [0,1]$, the second term in (\ref{calculus})
  is eqaul to $0$.
  Moreover, since the function $f_t|_{\partial D}$ is constant
  for any $t\in [0, 1]$,
  the first term in (\ref{calculus}) is equal to
  $S(g_t) + \int_0^1 (g_t^*(i_{X_t}\eta))(\gamma(1))dt$
  and the lemma follows.
\end{proof}

Put $\eta = (r^2 d\theta) /2$
and $\varphi_t = g_t|_{\partial D}$ in ${\rm Diff}_{+}(S^1)$.
Take a path $\gamma : [0, 1] \to D$ defined by $\gamma(t) = (t, 0)$.
Let $\widetilde{\varphi_t} \in \widetilde{{\rm Diff}_{+}(S^1)}$
be the lift of $\varphi_t$ such that $\widetilde{\varphi_0} = {\rm id}$.
As in the equation (\ref{i_X_t_eta_integral}), we have
\[
  \int_0^1 g_t^* (i_{X_t} \eta)(\gamma(1)) dt
  =
  \frac{1}{2} \int_0^1
  \frac{\partial \widetilde{\varphi_t}}{\partial t}(0) dt
  =
  \frac{1}{2} \widetilde{\varphi_1}(0),
\]
where we identify $\gamma(1) \in \partial D$
with $0 \in \mathbb{R}/2\pi\mathbb{Z}$
by the identification
$\partial D = S^1 = \mathbb{R}/2\pi \mathbb{Z}$.
Thus we have
\begin{equation}\label{sigma_S_rot}
  \sigma_{\eta, \gamma}(g_1) - S(g_t)
  =
  \frac{1}{2} \widetilde{\varphi_1}(0).
\end{equation}
This equality (\ref{sigma_S_rot})
implies that the value $S(g_t)$ depends only on the
homotopy class relatively to fixed ends of the path $g_t$
in $G_o$.
Henceforth, the map
$S: \widetilde{G_o} \to \mathbb{R} : [g_t] \mapsto S(g_t)$
is well-defined.
Moreover, the map $S$ gives rise to a homomorphism.
In fact, let $g_t, h_t$ be paths in $G_o$,
then
\begin{align*}
  &S(g_t h_t) - S(g_t) - S(h_t)\\
  &= \sigma_{\eta, \gamma}(g_1 h_1) -
  \sigma_{\eta, \gamma}(g_1) - \sigma_{\eta, \gamma}(h_1)
  -
  \frac{1}{2}(\widetilde{\varphi_1}\widetilde{\psi_1}(0)
  -\widetilde{\varphi_1}(0) - \widetilde{\psi_1}(0))
\end{align*}
and this is equal to $0$ (see Maruyama\cite{maruyama19}).
Thus we have
\begin{equation}\label{sigma_S_trans}
  \overline{\sigma}(g_1) - S([g_t])
  =
  \lim_{n \to \infty}\frac{\sigma_{\eta, \gamma}
  (g_1^n) - S([g_t]^n)}{n}
  =
  \pi \lim_{n\to \infty} \frac{\widetilde{\varphi_1}^n(0)}{2\pi n}
  = \pi \widetilde{\rm rot}(\widetilde{\varphi_1}).
\end{equation}
By the above equality (\ref{sigma_S_trans}),
we obtain the following theorem.

\begin{theo}\label{main_thm_2}
  Let $p: \widetilde{G_o} \to G_o$ be the projection.
  Then, we have
  \[
    p^*\overline{\sigma} - S = \pi \widetilde{\rm rot}
    : \widetilde{G_o} \to \mathbb{R}.
  \]
  Here the map $\widetilde{\rm rot}:\widetilde{G_o} \to \mathbb{R}$
  is the pullback of the translation number
  by the surjection
  $\widetilde{G_o} \to \widetilde{{\rm Diff}_{+}(S^1)}$.
\end{theo}

\begin{rema}
  By considering the map to $\mathbb{R}/\mathbb{Z}$,
  we obtain a theorem similar to Theorem
  \ref{theorem:tau_R_rotation}
  for $\overline{\sigma}$, $S$, and the rotation number.
\end{rema}

\begin{rema}
  By (\ref{sigma_S_rot}), we obtain the formula similar to
  \cite[Corollary (2.9)]{tsuboi00} and
  thus the formula similar to \cite[Proposition (3.1)]{tsuboi00}.
  This implies that the homomorphism ${\rm Flux}_{\mathbb{R}}$
  cannot be extended to a homomorphism on $G_o$.
\end{rema}

\section{Relation between $\overline{\tau}$ and $\overline{\sigma}$}
The restriction ${\rm Cal}|_{G_{o,{\rm rel}}} : G_{o,{\rm rel}} \to \mathbb{R}$
of the Calabi invariant
remains surjective.
So the restriction $\overline{\tau} : G_o \to \mathbb{R}$ is also
surjective homogeneous quasi-morphism.
Therefore we have two non-trivial homogeneous quasi-morphisms
$\overline{\tau}, \overline{\sigma} \in Q(G_o)$.
By Proposition \ref{tau_and_bdd_euler} and Proposition
\ref{sigma_bar},
the class $[\delta \overline{\tau}]$ coincides
with $\pi[\delta \overline{\sigma}]$
in $H_b^2(G_o;\mathbb{R})$.
Thus the difference
$\overline{\tau} - \pi \overline{\sigma}$ is a homomorphism
on $G_o$.
This implies that, in contrast with
${\rm Cal}$ and ${\rm Flux}_{\mathbb{R}}$,
the difference ${\rm Cal} - \pi {\rm Flux}_{\mathbb{R}}$ can be
extended to a homomorphism
$\overline{\tau} - \pi \overline{\sigma}:G_o \to \mathbb{R}$.

\begin{theo}\label{main_thm_3}
  The difference $\overline{\tau} - \pi \overline{\sigma}
  : G_o \to \mathbb{R}$ is a continuous surjective homomorphism.
\end{theo}

\begin{proof}
  On the group $G_{o,{\rm rel}}$, the homomorphism
  $\overline{\tau} - \pi \overline{\sigma}$ is equal to
  ${\rm Cal} - \pi {\rm Flux}_{\mathbb{R}}$.
  Put the non-increasing $C^{\infty}$-function
  $f : [0,1] \to \mathbb{R}$
  which is equal to $1$ near $r = 0$ and $f(1) = 0$.
  Then, for $s\in \mathbb{R}$, we define a diffeomorphism
  $g_s$ in $G_{o,{\rm rel}}$ by
  \[
    g_s(r, \theta) = (r, \theta + sf(r))
  \]
  where $(r, \theta) \in D$ is the polar coordinates.
  For
  \[
    \eta = (r^2 d\theta)/2,
    \hspace{5mm}
    \gamma(r) = (r, 0) \in D,
  \]
  we have
  \[
    {\rm Cal}(g_s) = \frac{s\pi}{2}\int_0^1 r^4
    \frac{\partial f}{\partial r}dr \  , \
    \pi {\rm Flux}_{\mathbb{R}}(g_s) = \frac{s\pi}{2}\int_0^1 r^2
    \frac{\partial f}{\partial r}dr.
  \]
  This implies that the difference
  $\overline{\tau} - \pi \overline{\sigma}$ is surjective on $G_{o,{\rm rel}}$,
  and so is on $G_o$.
\end{proof}

\bibliographystyle{amsplain}
\bibliography{samplebib}
\end{document}